\documentclass{amsart}

\usepackage[margin=1.00in]{geometry}
\usepackage{mathtools}
\usepackage{amsmath}
\usepackage{amsthm}
\usepackage{graphicx}
\usepackage{caption}
\usepackage{xcolor}
\usepackage{hyperref}
\usepackage{float}

\usepackage{comment}
\usepackage{mathrsfs}
\usepackage{enumitem}
\usepackage{dsfont}

\newtheorem{thm}{Theorem}
\newtheorem{mainthm}{Theorem}

\newtheorem{prop}{Proposition}
\newtheorem{lemma}{Lemma}
\newtheorem{cor}{Corollary}
\theoremstyle{remark}
\newtheorem*{remark}{Remark}

\newtheorem{exmp}{Example}

\theoremstyle{definition}

\newcommand{\lp}{\left(}
\newcommand{\rp}{\right)}

\newcommand{\lb}{\left[}
\newcommand{\rb}{\right]}

\newcommand{\al}{\alpha}
\newcommand{\be}{\beta}
\newcommand{\ga}{\gamma}
\newcommand{\de}{\delta}

\title[A Non-Autonomous Model for Parabolic Implosion]{A Non-Autonomous Model for Parabolic Implosion}
\author{Katelynn Huneycutt}
\address{Department of Mathematics, The Ohio State University, Columbus, OH 43210, USA}
\email{huneycutt.13@osu.edu}

\author{Samantha Sandberg-Clark}
\address{Department of Mathematics, The Ohio State University, Columbus, OH 43210, USA}
\email{sandberg-clark.1@osu.edu}

\author{Liz Vivas}
\address{Department of Mathematics, The Ohio State University, Columbus, OH 43210, USA}
\email{vivas.3@osu.edu}

\begin{document}

\begin{abstract}
Orthogonal polynomials appear naturally in the study of compositions of M\"obius transformations. In this paper, we consider several classes of orthogonal polynomials associated to non-autonomous perturbations of a parabolic M\"obius  map. Our results can be viewed as instances of non-autonomous parabolic implosion, including a random perturbative regime in which convergence holds almost surely.
\end{abstract}

\maketitle

\section{Introduction}

The study of parabolic implosion, starting with the work of Lavaurs \cite{Lavaurs1989}, has had striking consequences on the study of dynamics in one complex variable.  In recent years, this study has been extended quite successfully to several complex variables. 

Parabolic implosion in several variables has been developed in multiple settings. Starting with a map with a semi-attracting fixed point, Bedford, Smillie, and Ueda \cite{Bed16} described a parabolic implosion phenomenon, later generalized by Dujardin and Lyubich \cite{DuLy15}. Similarly Bianchi studied parabolic implosion for a class of maps on dimension $2$ tangent to the identity \cite{Bia19}.

Another setup in which parabolic implosion has been generalized, is to the case of skew-product maps that are tangent to the identity. Recent work by Astorg, Buff, Dujardin, Peters, Raissy \cite{ABD+16} exploit a parabolic implosion phenomenon in this setup to prove the existence of wandering Fatou components. Similar results have been obtained by Astorg and Boc-Thaler \cite{AstBoc22} as well as by Astorg and Bianchi \cite{AstBia24}.

Although our core estimates are one-dimensional, they are designed to be applied to skew-product dynamics in $\mathbb{C}^2$, thus describing new cases of parabolic implosion. We follow the approach of Vivas and study the map $F(z)=\frac{z}{1-z}$ and its perturbations.  As in \cite{Viv20}, we transform the problem to a multiplication of matrices and use orthogonal polynomial theory. We should point out that in our new setup we require a different class of orthogonal polynomials than the ones used in \cite{Viv20}. See section \ref{sec:orthogonalpolynomials} for more details.

Let us start by describing our results in more detail. The classical theorem of Lavaurs involves the perturbation of any given map $f$ as below.
 \begin{thm}[Lavaurs]
    Let $f$ be defined in a neighborhood $V$ of the origin and be of the form $f(z)=z+z^{2}+O(z^{3})$. Consider the perturbation of f as follows: given $\epsilon>0 $ let $f_\epsilon(z) \coloneqq f(z)+\epsilon^2$. If we take a sequence of numbers $N_\epsilon \to \infty$ and $\epsilon \to 0$, such that $N_\epsilon-\frac{\pi}{\epsilon}\rightarrow 0$. Then we obtain the following:
    \begin{align*}
        (f_\epsilon)^{N_\epsilon} \rightarrow \mathcal{L}_f
    \end{align*}
    uniformly on compacts on the basin of attraction of $f$. Here $\mathcal{L}_f$ is the Lavaurs map of $f$.
\end{thm} 

The Lavaurs map of $f$ is a transformation describing the discrepancy between the incoming and outgoing linearizing coordinates for $f$ on their respective incoming and outgoing basins.  We will refer to this phenomenon as \textit{autonomous parabolic implosion} for $f$.

\subsection{Non-autonomous perturbations} The natural question is what happens if the perturbation varies at
each step. Given a sequence $\{\epsilon_{k}\}_{k=1}^{N}$, define $f_{k}(z)\coloneqq f(z)+\epsilon_{k}^2$. We are interested in under which conditions on $\{\epsilon_{k}\}_{k=1}^{N}$ does 
\begin{align*}
     f_N\circ f_{N-1}\circ \cdots\circ f_2\circ f_1 \rightarrow  \mathcal{L}_{f}.
\end{align*}

This ``non-autonomous'' parabolic implosion shows up naturally when studying parabolic implosion in several variables. In effect, Astorg et al. \cite{AstBia24,ABD+16,AstBoc22} have proven results of this type.

In \cite{Viv20}, Vivas focuses on the case of the initial map to be $f(z)=z/(1-z)$ since in this case each perturbation is also a M\"obius transformation, and the Lavaurs map is the identity. 

\subsection{Our results} 

We extend the results of \cite{Viv20} by allowing both multiplicative and additive perturbations of the model map.  
Specifically, given sequences $\{\rho_{k},\epsilon_{k}\}_{k=1}^{N}$, we consider
\[
   f_{k}(z):=\rho_k\frac{z}{1-z}+\epsilon_{k}^2,
   \qquad \rho_k = e^{2\pi i \theta_k}, \ \theta_k\in\mathbb{C},
\]
and ask under what conditions the composition
\[
   f_N\circ f_{N-1}\circ \cdots\circ f_1
\]
converges to the identity.  

Our main theorems show that this convergence persists under several regimes:
\begin{itemize}
   \item \textit{Purely multiplicative perturbations ($\epsilon_k=0$).} Convergence holds provided the sequence $\{\theta_k\}$ approximates $1/N$ closely enough, with controlled error terms (Theorem~\ref{thm:A}).
   \item \textit{Combined multiplicative and additive perturbations.} More general families of convergent sequences arise when $\{\theta_k\}$ and $\{\epsilon_k\}$ are simultaneously perturbed (Theorem~\ref{thm:B}).
   \item \textit{Random additive perturbations ($\rho_k=1$).} If the additive errors fluctuate randomly with mean $\frac{\pi}{N}$, then convergence still holds almost surely (Theorem~\ref{thm:random}). 
\end{itemize}

These results highlight the delicate balance in the non-autonomous setting.  
Errors of order $1/N^2$ already require strong correlation conditions, but introducing randomness allows one to relax these restrictions.  
Note also that although multiplicative and additive perturbations may look analogous, they are not directly reducible to each other: in the additive case each perturbation can be conjugated to a rotation, but the conjugacy varies with $k$, so the multiplicative case cannot be deduced from the additive one.  

In fact, we will see later that the rotation case ($\rho_k\neq 1$) cannot be deduced directly from the additive perturbative case (i.e., $\rho_k=1$).

Let us state our theorems now:

\begin{mainthm}[Multiplicative perturbations]\label{thm:A}
Let
\[
   f_k(z):= \frac{\rho_k z}{1-z}, 
   \qquad \rho_k = e^{2\pi i\theta_k},\qquad \theta_k\in\mathbb{C}.
\]
Suppose $\{\theta_k\}$ satisfies any of the following:
\begin{enumerate}
   \item $\displaystyle \theta_k = \frac{1}{N} + O\!\left(\tfrac{1}{N^3}\right)$;
   \item $\displaystyle \theta_k = \frac{1}{N} + \frac{c_k}{N^2}$ with $c_k$ uniformly bounded and $\; c_k+c_{k+1} = O\!\left(\tfrac{1}{N}\right)$ for odd $k$;
   \item $\displaystyle \theta_k = \frac{1}{N} + \frac{Ce^{2\pi i k/N}}{N^2}$ where $C$ is any constant in $\mathbb{C}$.
\end{enumerate}
Then 
\[
   f_N\circ f_{N-1}\circ \cdots\circ f_1 \longrightarrow \mathrm{Id}
\]
as $N\to\infty$, uniformly on compact subsets of $\mathbb{C}$.
\end{mainthm}

\begin{mainthm}[Combined multiplicative and additive perturbations]\label{thm:B}
Let
\[
   f_k(z) := \frac{\rho_k z}{1-z} + \epsilon_k^2, 
   \qquad \rho_k = e^{2\pi i \theta_k},\ \ \theta_k\in\mathbb{C}.
\]
Suppose \(\{\theta_k\}\) and \(\{\epsilon_k\}\) satisfy any of the following:
\begin{enumerate}
   \item \(\displaystyle \theta_k = \frac{1}{N} + O\!\left(\tfrac{1}{N^3}\right)\) and \(\displaystyle \epsilon_k = O\!\left(\tfrac{1}{N^2}\right)\);
   \item \(\displaystyle \theta_k = \frac{1}{N} + \frac{c_k}{N^2}\) with $c_k$ uniformly bounded and \(\, c_k+c_{k+1} =O\!\left(\tfrac{1}{N}\right)\) for odd \(k\), and \(\displaystyle \epsilon_k = O\!\left(\tfrac{1}{N^2}\right)\);
  \item $\displaystyle \theta_k = \frac{1}{N} + \frac{Ce^{2\pi i k/N}}{N^2}$  and \(\displaystyle \epsilon_k = O\!\left(\tfrac{1}{N^2}\right)\); where $C$ is any constant in $\mathbb{C}$.
   \item \(\displaystyle \theta_k = 0\) and \(\displaystyle \epsilon_k = \frac{\pi}{N} + O\!\left(\tfrac{1}{N^3}\right)\);
   \item \(\displaystyle \theta_k = 0\) and \(\displaystyle \epsilon_k = \frac{\pi}{N} + \frac{c_k}{N^2}\) with \(\,c_k + c_{N-k} = O\!\left(\tfrac{1}{N}\right)\).
\end{enumerate}
Then
\[
   f_N\circ f_{N-1}\circ \cdots\circ f_1 \longrightarrow \mathrm{Id}
\]
as \(N\to\infty\), uniformly on compact subsets of $\mathbb{C}$.
\end{mainthm}

\begin{remark}
Conditions (2)–(3) show that errors of order \(1/N^2\) in the multiplicative phase \(\theta_k\) are admissible provided they satisfy summation/cancellation constraints, while the additive terms remain at scale \(1/N^2\). Conditions (4)–(5) cover the purely additive regime near the critical scale \(\pi/N\) familiar from Lavaurs theory, again allowing \(1/N^2\)-level modulations subject to symmetry constraints.
\end{remark}

We also will prove that the multiplicative case is inherently different than the additive one. We do that by exhibiting an example in which one sequence composition converges to the identity; however, a conjugate sequence, does not. 

\begin{mainthm}[Difference between additive and multiplicative]\label{thm:C}
There exist sequences $\{\epsilon_k\}$ and $\{\rho_k\}$ for $1\leq k\leq N$, such that:
\begin{enumerate}
\item For each $k$, the maps $\displaystyle{f_k(z)=\frac{\rho_k z}{1-z}}$ and $\displaystyle{g_k(z)=\frac{z}{1-z}+\epsilon_k^2}$ are conjugated.
\item We have $g_N \circ g_{N-1} \circ\cdots \circ g_1$ converges uniformly on compacts to the Identity.
\item The composition $f_N \circ f_{N-1}\circ\cdots \circ f_1(z)$ does not converge to $z$ for any $z$.
\end{enumerate}
\end{mainthm}

Finally, we study what happens if we apply a random composition. We can weaken the conditions if we allow for a sequence of independent and bounded random variables with $0$ mean as the scaling factors of $1/N^2$. More explicitly:

\begin{mainthm}[Random perturbations]\label{thm:random}
For fixed \(N\), let \(\{\eta_{k}\}_{k=1}^{N+1}\) be independent, bounded, mean-zero random variables, and set
\[
   \epsilon_{k} = \frac{\pi}{N}+\frac{\eta_{k}}{N^{1+\delta}}\quad (\delta>0), 
   \qquad
   f_k(z) = \frac{z}{1-z} + \epsilon_k^2 .
\]
Then
\[
   f_N\circ f_{N-1}\circ \cdots\circ f_1 \longrightarrow \mathrm{Id}
\]
as \(N\to\infty\), uniformly on compact subsets of $\mathbb{C}$, with probability one.
\end{mainthm}

\begin{remark}
The random theorem shows that decreasing the error scale to \(N^{-(1+\delta)}\) (any \(\delta>0\)) and imposing only boundedness and zero mean suffices for almost sure convergence, without the arithmetic correlations required in the deterministic \(1/N^2\)-scale cases.
\end{remark}

The paper is organized as follows: in Section \ref{sec:compositionoflineartransformations} we link iterations of a sequence of linear transformations to orthogonal polynomials. 
In Section \ref{sec:orthogonalpolynomials}, we review the relevant results on orthogonal polynomials.  
In Section \ref{sec:nonautonomousimplosions}, we prove Theorems \ref{thm:A} and \ref{thm:B}. 
In Section \ref{sec:differencewiththeadditive} we discuss the differences between the additive and multiplicative case, as stated in Theorem \ref{thm:C}. In Section \ref{sec:bifurcationsforskew-products} we give some examples of skew parabolic maps on which we can apply Theorems \ref{thm:A} and \ref{thm:B} . 
We prove Theorem \ref{thm:random} in Section \ref{sec:randomcompositions}.

\section{Compositions of Linear Transformations}\label{sec:compositionoflineartransformations}

In this section we describe how the compositions of linear fractional transformations can be expressed in terms of orthogonal polynomials.

\begin{lemma}\label{generalFormComposition}
Given \[
f_k(z)= \frac{\al_k z+\be_k}{\ga_k z +\de_k},\] an infinite sequence of non-trivial linear fractional transformations for $k\geq 1$ (i.e., assume $\al_k\de_k\neq \be_k\ga_k$ for all $k$). Assume that $\gamma_k \neq 0$ for all $k$.

Denote by $F_N:=f_N \circ f_{N-1} \circ \cdots \circ f_1$. Then we have the following formula:
\begin{align*}
F_N(z) = \frac{A_N z + B_N}{C_N z +D_N},
\end{align*}
where the coefficients are given by the following recurrence formulas:
\begin{align}
C_{k+1} &= \left(\de_{k+1}+\frac{\ga_{k+1}\al_{k}}{\ga_{k}}\right)C_k - \frac{\ga_{k+1}}{\ga_{k}}\left(\al_k\de_k-\be_k\ga_k\right)C_{k-1}, & C_0 &= 0, & C_1&= \ga_1;& \nonumber \\
D_{k+1} &= \left(\de_{k+1}+\frac{\ga_{k+1}\al_{k}}{\ga_{k}}\right)D_k - \frac{\ga_{k+1}}{\ga_{k}}\left(\al_k\de_k-\be_k\ga_k\right)D_{k-1}, & D_0 &= 1, & D_1&= \de_1;& \label{orthogonal}
\end{align}
and
\begin{align} \label{orthogonal2}
A_{k} &= \frac{1}{\ga_{k+1}}(C_{k+1}-\de_{k+1}C_{k});\nonumber\\
B_{k} &= \frac{1}{\ga_{k+1}}(D_{k+1}-\de_{k+1}D_{k}) .
\end{align}
\end{lemma}

\begin{proof}
For all $k$, by the definition of $C_{k}$ and $D_{k}$
\begin{align*}
C_{k+1} &= \gamma_{k+1}A_{k}+\delta_{k+1}C_{k};\\
D_{k+1} &= \gamma_{k+1}B_{k}+\delta_{k+1}D_{k}.
\end{align*} 
By solving for $A_{k}$ and $B_{k}$, we find \eqref{orthogonal2}. 
The formulas for $C_{k+1}$ and $D_{k+1}$ follow by writing the definitions of $C_{k+1}$ and $D_{k+1}$, substituting in the definition of $A_{k}$ and $B_{k}$ respectively, then finally replacing the new $A_{k-1}$ and $B_{k-1}$ terms using the formulas from 
equations \eqref{orthogonal2}. Rearranging the remaining terms yields the result.
\end{proof}

\begin{remark}\mbox{}
\begin{enumerate}
\item Note that we can also obtain expressions for $A_k$ and $B_k$ as recurrence equations, similarly as in equation \eqref{orthogonal}.
\item Since there is a degree of freedom as to how we write each fractional transformation, we can always normalize them so that one coefficient on $f_k$ is a given constant (unless that coefficient is $0$ to begin with).  In the following we will normalize each matrix so that all $\ga_k$ are equal whenever $\gamma_k\neq 0$ for all $k$.
\item Now if we make each $\displaystyle{f_k(z) =\frac{z}{1-z}+\epsilon^2}$, i.e. $\al_k=1-\epsilon^2, \be_k=\epsilon^2, \ga_k=-1, \de_k=1$ for all $k$ then we obtain the following equations for $C_k$ and $D_k$:
\begin{align*}
    C_{k+1} &= (2-\epsilon^2)C_k - C_{k-1}, & C_0 &= 0, & C_1&= -1; &\\
    D_{k+1} &= (2-\epsilon^2)D_k - D_{k-1}, & D_0 &= 1, & D_1&= 1.&
\end{align*}
As noted in \cite{Viv20}, these are precisely the classical Chebyshev polynomials, linking
the iteration of these transformations to orthogonal polynomial theory.
\end{enumerate}
\end{remark}

We apply now Lemma \ref{generalFormComposition} to the family of maps that was mentioned in the introduction. 

\begin{lemma}\label{lem:recursivepoly}
Let $(\epsilon_k,\rho_k)_{k \geq 1}$ be sequences of complex numbers. Consider the 
maps $\displaystyle{f_{k}(z) = \frac{\rho_{k}z}{1-z}+\epsilon_{k}^2}$. 
Denote  by $F_N:=f_N \circ f_{N-1} \circ \cdots \circ f_1$ as above. Then: 
 \begin{align*}
F_N(z) = \frac{A_N z + B_N}{C_N z +D_N},
\end{align*}
where the coefficients are given by the following recurrence formulas:
\begin{align*}
A_{k} &= q_{k+1}-q_{k};\\
B_{k} &= r_{k}-r_{k+1}; \\
C_{k} &= -q_{k};\\
D_{k} &= r_{k};
\end{align*}
for:
\begin{align*}
q_{k+1} &= (1+\rho_k-\epsilon_{k}^{2}) q_k-\rho_k q_{k-1}, & q_0&=0, & q_1&=1;& \nonumber \\
r_{k+1} &= (1+\rho_k-\epsilon_{k}^{2}) r_k-\rho_k r_{k-1},& r_0&=1, & r_1&=1. &    
\end{align*}
\end{lemma}

\begin{proof}
This is the specialization of Lemma~\ref{generalFormComposition} to $\alpha_k=\rho_k$, $\beta_k=\epsilon_k^2$, $\gamma_k=-\rho_k$, $\delta_k=1$, with a straightforward computation.
\end{proof}

When all $\epsilon_k$ are equal to $0$, we obtain:

\begin{cor}\label{cor:onlyrho}
Let $(\rho_k)_{k \geq 1}$ sequences of complex numbers. Consider the 
maps $\displaystyle{f_{k}(z) = \frac{\rho_{k}z}{1-z}}$. 
Denote  by $F_N:=f_N \circ f_{N-1} \circ \cdots \circ f_1$ as above. Then: 
 \begin{align*}
F_N(z) = \frac{A_N z}{C_N z +1},
\end{align*}
where the coefficients are given by the following recurrence formulas:
\begin{align*}
A_{k} &= q_{k+1}-q_{k};\\
C_{k} &= -q_{k};
\end{align*}
where:
\begin{align*}
q_{k+1} &= (1+\rho_k) q_k-\rho_k q_{k-1}, & q_0&=0, & q_1&=1. &
\end{align*}
\end{cor}

In the next section we focus on the detailed study of these recursions.

\section{Orthogonal Polynomials}\label{sec:orthogonalpolynomials}

We now study the orthogonal polynomials arising from the recursive sequences in Lemma \ref{lem:recursivepoly}.
Recall the sequence:
\begin{align*}
q_{k+1} &= (1+\rho_k-\epsilon_k^2) q_{k}-\rho_k q_{k-1}, & q_0&=0,& q_1&=1, &
\end{align*}
where the coefficients depend on the sequences $(\rho_k,\epsilon_k)$.
Our goal is to understand, under which conditions on $\rho_k$ and $\epsilon_k$, the value $q_N$ tends to $0$ as $N$ tends to infinity.
	
	\textbf{Past work}: The case $\rho_k \equiv 1$
 was studied in detail by Vivas \cite{Viv20}, who obtained conditions on perturbations of the form  $\epsilon_k = \pi/N + \alpha(k)/N^2$ to ensure $q_N = O(1/N)$. The main tool was to compare $q_k$ with the Chebyshev polynomials: the sequence $U_k$, given by $U_{k+1}=(2-\epsilon^2)U_k-U_{k-1}$.
	
When the $\rho_k$ are allowed to vary on the unit circle or close to it, however, this approach no longer works. Comparing to the Chebyshev polynomials $U_k$ in this case does not lead to useful bounds and do not reflect the structure of the new recurrence.
	
\textbf{New sequence}: 	With that in mind, we compare each $q_{k}$ with the recursive family arising in the autonomous rotation case (see Corollary \ref{cor:onlyrho}), that is where each $\rho_k = \rho$. Denote this sequence by $T_{k}$, defined by:
\begin{align*}
T_{k+1} &= (1+\rho) T_{k}-\rho T_{k-1}, & T_0&=0, & T_1&=1. &
\end{align*}
By induction, we can show
\begin{align*}
T_{k} = 1+ \rho +\rho^{2} + \dots +\rho^{k-1} = \frac{1-\rho^k}{1-\rho}.
\end{align*}
In particular, when $\rho=e^{2\pi i/N}$ we can rewrite:
\begin{align}
T_k = e^{\frac{i\pi (k-1)}{N}}\frac{\sin(\pi k/N)}{\sin(\pi/N)}. \label{eq: trigT_k}
\end{align}
    From this formula, we will frequently reference that $T_N=0$ and $T_{N+1}=1$, as illustrated in Figure~\ref{fig:tklargemodulus}, and that $|T_k|$ is maximized when $k=\tfrac{N}{2}$, with the maximum approximately $N/\pi$ when $N$ is large.

\begin{figure}[H]
  \centering
  \begin{minipage}[b]{0.48\linewidth}
    \centering
    \includegraphics[height=6cm]{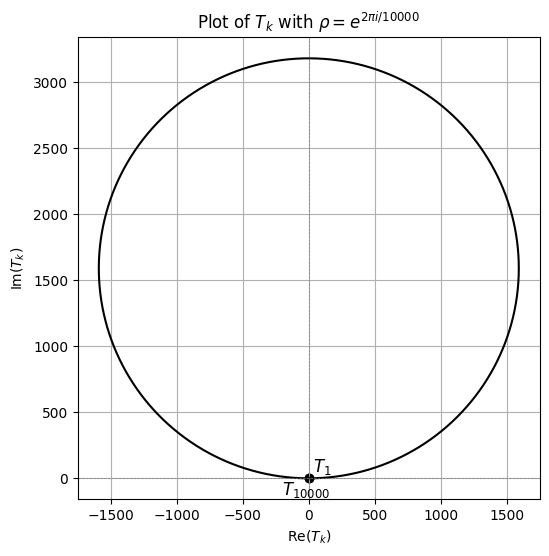}
  \end{minipage}
  \hfill
  \begin{minipage}[b]{0.48\linewidth}
    \centering
    \includegraphics[height=6cm]{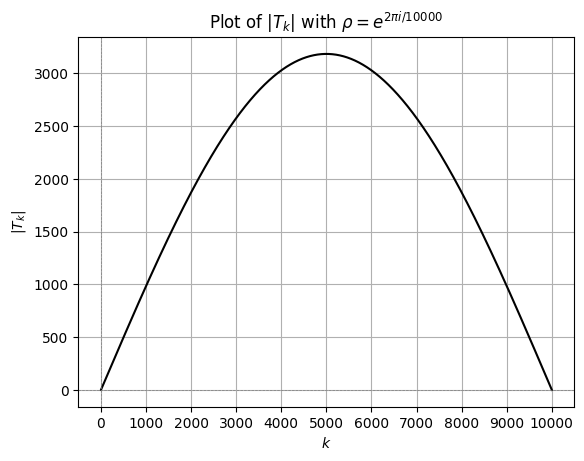}
  \end{minipage}
  \caption{Graph of $T_k$ and $|T_k|$ for $N=10,000$.}
  \label{fig:tklargemodulus}
\end{figure}	

    As a first application we observe the following example:
	
	\begin{exmp}\label{exmp: basic}
		Let $\displaystyle{f_{k}(z) = \frac{\rho z}{1-z}}$ for all $k\geq 1$, for $\rho=e^{2\pi i/N}$ for a fixed $N$. Then by Corollary \ref{cor:onlyrho} and equation \eqref{eq: trigT_k}, the coefficients of $F_{N}$ are given by 
		\begin{align*}
			A_{N} = T_{N+1}-T_{N} = 1,\ B_{N} = 0,\ C_{N} = -T_{N} = 0,\ D_{N}= 1.
		\end{align*} 
		Thus, $F_{N}(z)= z$.
	\end{exmp}	

We use the orthogonal polynomials ${T_k}$ to derive meaningful bounds for the sequence ${q_k}$. This is achieved through a formula that expresses the difference between ${q_k}$ and ${T_k}$. We first state this result in a general form by considering any sequences of  $a_k,$ and $b_{k}$ in $\mathbb{C}$ as additive perturbations to the coefficients in the recursive formula.   

	\begin{lemma}\label{lem:relationship}
		Consider the recursive polynomials 
		\begin{align}
			T_{k+1} &= (1+\rho) T_{k}-\rho T_{k-1}, & T_0&=0, & T_1&=1;& \nonumber \\ 
			q_{k+1} &= (1+\rho+a_k) q_{k}-(\rho+b_k) q_{k-1}, & q_0&=0, & q_1&=1.  & \label{eq:qk formula}
		\end{align}
		Their difference is given by 
		\begin{align*}
			q_k-T_k = \sum_{j=1}^{k-1}[a_jq_j-b_jq_{j-1}]T_{k-j}, \quad k\geq 2.
		\end{align*}
	\end{lemma}
    \begin{proof}
    This follows by induction by subtracting the two relationships. 
    \end{proof}

	\begin{remark}
   In our applications of this lemma, the sequence ${b_k}$ represents how far our multiplicative rotation deviates from the root of unity $e^{2\pi i/N}$; that is, $b_k = \rho_k - \rho$. The sequence ${a_k}$ is defined by $a_k = b_k - \epsilon_k^2$.

In the case where $\epsilon_k \equiv 0$ for all $k$, the expression for $a_{k}$ simplifies to $a_k = b_k$. Otherwise, the sequence ${a_k}$ incorporates an additional perturbation term $\epsilon_k^2$.
    \end{remark}

We will use the estimate above repeatedly in the following lemmas.

\subsection{Quadratic errors}

Fix $\rho=e^{2\pi i/N}$. Our goal is to identify conditions on $a_k$ and $b_k$ that guarantee the errors \[|q_N|=|q_N-T_N|\quad\text{and}\quad|q_{N+1}-1|=|q_{N+1}-T_{N+1}|\] remain small. 

Note that if we suppose $a_{k}=b_{k}$ and 
\[
\rho+b_k=\rho+a_k=e^{2 \pi i/(N+1)},
\]
then $a_k=b_k=O(1/N^2)$. We will refer to such a choice of $a_{k}$ and $b_{k}$ as quadratic or the resulting orthogonal polynomials as having quadratic errors.  In this case, the recursion yields
\[
q_{k+1}=(1+e^{2 \pi i/(N+1)})q_k-e^{2 \pi i/(N+1)}q_{k-1},
\]
which allow us to explicitly compute the formula $q_k= e^{\frac{i\pi (k-1)}{N+1}}\frac{\sin(\pi k/(N+1))}{\sin(\pi/(N+1))} $ and the values:
\[|q_N|=1\quad\text{and}\quad q_{N+1}=0.\]

In particular, the mere condition $b_k,a_k = O(1/N^2)$ does not suffice to ensure that $|q_N-T_N|=o(1)$.

Nevertheless, when $a_k$ and $b_k$ are quadratic and comparable, we can still establish a uniform control on successive differences.

\begin{lemma}\label{lem:q_kdifferencebound}
    Fix $N\in\mathbb{N}$ and set $\rho=e^{2\pi i/N}$. Consider the recurrence \eqref{eq:qk formula}
    \begin{align*}
        q_{k+1} &= (1+\rho+a_k) q_{k}-(\rho+b_k) q_{k-1}, & q_0&=0, & q_1&=1, &
    \end{align*}
    for $1\leq k\leq N+1$.
    Assume there is a constant $C>0$, independent of $N$, such that for all $k$
\[
a_k=\frac{c_k}{N^2}+E_{1,k},\qquad b_k=\frac{c_k}{N^2}+E_{2,k},
\]
with $|c_k|\le C$ and $|E_{j,k}|\le C/N^4$ ($j=1,2$). 

    Then:
    \begin{itemize}
    \item There exists a constant $\tilde{C}>0$, independent of both $k$ and $N$, such that
    \[
        |q_k - q_{k-1}| \leq \tilde{C} \quad \text{for all $1\leq k\leq N+1$}.
    \]
    \item There exists a constant $C'>0$, independent of $k$ and $N$, such that 
    \begin{equation*}
        |q_k-T_k|\leq C'  \quad \text{for all $1\leq k\leq N+1$}.
    \end{equation*}
    \end{itemize}
\end{lemma}

\begin{proof}
    Let $\tilde{C}>2$ satisfy $\frac{\max_\ell |c_\ell|}{N}\cdot \tilde{C} + \frac{4C}{N}<\tilde{C}$ and work using induction. 
    First, note that $|a_j|<\epsilon$ for some $0<\epsilon<1$.
    For the base case, observe that $q_0= 0$, $q_1=1$, and $q_2=1+\rho+a_1$ from which it follows $|q_0|=0\leq 2N$, $|q_1|=1\leq 2N$, and $|q_2|\leq 1+\epsilon < 2N$.
    We also have 
    \begin{align*}
        |q_1-q_0| &= 1 < \tilde{C}, \\
        |q_2-q_1| &= | \rho+a_1| \leq 1 + \epsilon < \tilde{C}.
    \end{align*}

    Step 1: $|q_j|\leq 2N$ for all $1\leq j\leq k$ implies $|q_{k+1}-q_k|\leq \tilde{C}$.
    Now suppose $|q_j|\leq 2N$ and $|q_j-q_{j-1}|\leq \tilde{C}$ for all $1\leq j\leq k$.
    Then using Lemma \ref{lem:relationship} and induction,
    \begin{align*}
        |q_{k+1}-q_k| 
        &= |(q_{k+1}-T_{k+1}) + (T_{k+1}-T_k) + (T_k - q_k)| \\
        &= \left| \sum_{j=1}^k [a_jq_j-b_jq_{j-1}]T_{k+1-j} + \rho^k - \sum_{j=1}^{k-1} [a_jq_j-b_jq_{j-1}]T_{k-j} \right| \\
        &= \left| \sum_{j=1}^k[a_jq_j-b_jq_{j-1}] \rho^{k-j} + \rho^k \right| \\
        &= \left| \sum_{j=1}^k \lb \frac{c_j}{N^2} (q_j-q_{j-1}) \rb\rho^{k-j} + \sum_{j=1}^k [E_1q_j - E_2 q_{j-1}] +\rho^k \right| \\
        &\leq N \cdot \frac{\max_\ell|c_\ell|}{N^2}\cdot \tilde{C} + \sum_{j=1}^k\lp \frac{C}{N^4}\cdot |q_j| + \frac{C}{N^4}\cdot |q_{j-1}| \rp \\
        &\leq \frac{\max_\ell |c_\ell|}{N}\cdot \tilde{C} + \frac{4C}{N^2} \\
        &< \tilde{C}.
    \end{align*}

    Step 2: 
    Under the same assumptions, $|q_j|\leq 2N$ and $|q_j-q_{j-1}|\leq \tilde{C}$ for all $1\leq j\leq k$, and Lemma \ref{lem:relationship},
    \begin{align*}
        |q_{k+1}| &= \left|T_{k+1} + \sum_{j=1}^k(a_jq_j-b_jq_{j-1}) T_{k+1-j} \right| \\
        &= \left| T_{k+1} + \sum_{j=1}^k \lb\frac{ c_j}{N^2}(q_j-q_{j-1}) + E_1q_j - E_2 q_{j-1} \rb T_{k+1-j} \right| \\
        &\leq N + N\cdot \lb \frac{  \max_\ell|c_k|}{N^2}\cdot\tilde{C} + \frac{2C}{N^2} \rb\cdot N \\
        &= N + \max_\ell|c_\ell| \tilde{C} + 2C \\
        &< 2N,
    \end{align*}
    for sufficiently large $N$.

    Thus we conclude $|q_j|\leq 2N$ and $|q_j-q_{j-1}|\leq \tilde{C}$ for all $1\leq j\leq k+1$.
    By induction, $|q_j|\leq 2N$ and $|q_j-q_{j-1}|\leq \tilde{C}$ for all $1\leq j\leq N+1$.
\end{proof}

\subsection{Cubic errors}

We use now the estimates above when the errors are cubic, i.e. $a_{k},b_{k} = O\lp\frac{1}{N^{3}}\rp$.

\begin{prop}\label{prop:cubic}
Let $N\in\mathbb{N}$ and set $\rho = e^{2\pi i/N}$.  
Consider the recurrence in \eqref{eq:qk formula}:
	\begin{align*}
		q_{k+1} &= (1+\rho+a_k) q_{k}-(\rho+b_k) q_{k-1}, & q_0&=0, & q_1&=1, &
	\end{align*}
	where for each $1\leq k\leq N+1$: 
    \begin{align*}
        a_k =\frac{c_k}{N^3}+O\lp\frac{1}{N^4}\rp, \qquad b_k = \frac{c_k}{N^3}+O\lp\frac{1}{N^4}\rp
    \end{align*}
    with $c_k\in\mathbb{C}$ and $|c_k|<C$ for some constant $C>0$ independent of $N$ and $k$, and the $O(\tfrac{1}{N^4})$ bounds uniform in $k$ and $N$.
    Then there exists a constant $C'>0$, independent of $N$, such that
	\begin{equation*}
		|q_N| \leq \frac{C'}{N} \quad\text{and}\quad |q_{N+1}-1|\leq \frac{C'}{N}.
	\end{equation*}
\end{prop}
\begin{proof}
	Using Lemma \ref{lem:relationship} we obtain
	\begin{align}
		q_k-T_k &= \sum_{j=1}^{k-1} [a_jq_j - b_jq_{j-1}] T_{k-j} \nonumber \\
		&= \sum_{j=1}^{k-1} \left[ \lp\frac{c_j}{N^3}+ {O\lp\frac{1}{N^4}\rp} \rp q_j - \lp \frac{c_j}{N^3}+ {O\lp\frac{1}{N^4}\rp} \rp q_{j-1} \right] T_{k-j} \nonumber \\
		&= \sum_{j=1}^{k-1} \frac{c_j}{N^3}(q_j-q_{j-1}) T_{k-j} + \sum_{j=1}^{k-1}\lb {O\lp \frac{1}{N^4}\rp} q_j- O\lp\frac{1}{N^4}\rp q_{j-1}\rb T_{k-j}.\label{eq:qk-tk summands}
	\end{align}
    Using Lemma \ref{lem:q_kdifferencebound}, $|q_j-q_{j-1}|\leq \tilde{C}$ for some $\tilde{C}>0$ and $|q_j|\leq 2N$ for all $1\leq j\leq N+1$. Therefore for $k\leq N+1$ equation \eqref{eq:qk-tk summands} becomes 
	\begin{align*}
		|q_k-T_k| &= \left|\sum_{j=1}^{k-1}\frac{c_j}{N^3}(q_j-q_{j-1})T_{k-j} +  \sum_{j=1}^{k-1}\lb {O\lp \frac{1}{N^4}\rp} q_j- O\lp\frac{1}{N^4}\rp q_{j-1}\rb T_{k-j} \right| \\
		&\leq  \sum_{j=1}^{k-1}\left|\frac{c_j}{N^3}\right||q_j-q_{j-1}| |T_{k-j} |+ \sum_{j=1}^{k-1} \lb O\lp\frac{1}{N^4}\rp |q_j| + O\lp\frac{1}{N^4}\rp |q_{j-1}|\rb |T_{k-j}|\\
		&\leq  \sum_{j=1}^{k-1}  \frac{|c_j|}{N^3} \tilde{C} N + \sum_{j=1}^{k-1} O\lp\frac{1}{N^3}\rp  N \\
		&\leq O\lp\frac{1}{N}\rp.
	\end{align*}
	
	Since $T_N=0$ and $T_{N+1}=1$ we know
	\begin{align*}
		|q_N| = |q_N-T_N| =O\lp\frac{1}{N}\rp \qquad\textrm{and}\qquad
        |q_{N+1}-1| = |q_{N+1}-T_{N+1}| = O\lp \frac{1}{N}\rp.
        \end{align*}
\end{proof}

\subsection{Explicit condition for convergence}

If we want the error to be of quadratic order, we have to impose some conditions on it. We have the following:

\begin{prop}\label{pro:summationconstant}
	Let $N\in\mathbb{N}$ and set $\rho = e^{2\pi i/N}$.  
Consider the recurrence in \eqref{eq:qk formula}:
	\begin{align*}
		q_{k+1} &= (1+\rho+a_k) q_{k}-(\rho+b_k) q_{k-1}, & q_0 &=0, & q_1&=1, &
	\end{align*}
where for each $1\leq k\leq N+1$: 
\begin{align*}
{a}_k = \frac{c_k}{N^2}+O\lp\frac{1}{N^4}\rp,\ {b}_k = \frac{c_k}{N^2}+O\lp\frac{1}{N^4}\rp,    
\end{align*}
where all $c_k\subset\mathbb{C}$ are uniformly bounded in $N$, and the $O(\tfrac{1}{N^4})$ terms uniform in $k$ and $N$.  
If, in addition, there exists a constant $C>0$, independent of $N$, such that
\begin{equation}\label{eq: ak/N^2 sum requirement}
    \left|\sum_{k=1}^{N-1} a_kT_k \right| \leq \frac{C}{N}
\end{equation}
then there exists $C'>0$, independent of $N$, such that
\begin{equation*}
		|q_N| \leq \frac{C'}{N} \quad\text{and}\quad |q_{N+1}-1|\leq \frac{C'}{N}.
	\end{equation*}
\end{prop}

\begin{proof}
	Suppose ${a}_k = \frac{c_k}{N^2} +O\lp \frac{1}{N^4}\rp$ and ${b}_k = \frac{c_k}{N^2} + O\lp \frac{1}{N^4}\rp$, so $b_k = a_k + O(1/N^4)$. 
	Then observe that Lemma \ref{lem:relationship} gives
	\begin{align*}
		q_k-T_k &= \sum_{j=1}^{k-1} [{a}_jq_j - {b}_jq_{j-1}] T_{k-j} \\
		&= \sum_{j=1}^{k-1} a_j(q_j-q_{j-1}) T_{k-j} + \sum_{j=1}^{k-1} (a_j-b_j)q_{j-1} T_{k-j}.
	\end{align*}
	Taking the second summand in modulus, we have
	\begin{equation*}
		\left|\sum_{j=1}^{k-1} (a_j-b_j)q_{j-1} T_{k-j}\right|=
        \left|\sum_{j=1}^{k-1} O(1/N^4)q_{j-1} T_{k-j} \right| \leq O(1/N). 
	\end{equation*}
	Hence, in what follows, we work exclusively with the first summation $\sum_{j=1}^{k-1} a_j(q_j-q_{j-1})T_{k-j}$.
	
	Let $k=N$ and observe that \eqref{eq: trigT_k} gives
    \begin{align*}
    T_{N-j}&=e^{i\pi(N-j-1)/N}\frac{\sin((N-j)\pi/N)}{\sin(\pi/N)}\\
    &=-e^{-i\pi(j+1)/N}\frac{\sin(j\pi/N)}{\sin(\pi/N)}.
    \end{align*}
	We also have 
	\begin{align*}
		q_j-q_{j-1} 
		&= (\rho+{a}_{j-1})q_{j-1} -(\rho+{b}_{j-1})q_{j-2} 
		= (\rho+{a}_{j-1})(q_{j-1}-q_{j-2}) + O\lp \frac{1}{N^4}\rp q_{j-2}.
\end{align*}
We first prove the estimate under the simplifying assumption that $a_k=b_k$. We then show that $b_k-a_k=O(1/N^4)$ will contribute a sufficient small order term. Since $a_k=O(1/N^2)$ we may choose a branch of logarithm for $\rho+a_k$. Then:
        \begin{align*}
         		q_j-q_{j-1} 
	 	&= \prod_{k=1}^{j-1} (\rho+{a}_{k}) 
		= \prod_{k=1}^{j-1} e^{(2\pi i/N + c_k/N^2 + O(1/N^4))}\\
	 	&= \exp\left\{\sum_{k=1}^{j-1} \frac{2\pi i}{N} + \frac{c_k}{N^2} + O\lp\frac{1}{N^4}\rp \right\} \\
		&= \exp\left\{ \frac{2\pi i(j-1)}{N} + \frac{1}{N^2}\sum_{k=1}^{j-1} c_k + O\lp\frac{1}{N^3}\rp \right\}.
	 \end{align*}

    Then
    \begin{align*}
        \left|\sum_{j=1}^{N-1} a_j (q_j-q_{j-1})T_{N-j} \right|
        &= \left| \sum_{j=1}^{N-1} a_j e^{2\pi i(j-1)/N + \sum_{k=1}^{j-1}c_k/N^2 + O(1/N^3)} e^{-i\pi(j+1)/N}\frac{\sin(\pi j/N)}{\sin(\pi/N)} \right| \\
        &= \left| \sum_{j=1}^{N-1} a_j e^{\pi i(j-1)/N}\frac{\sin(\pi j/N)}{\sin(\pi/N)} e^{\sum_{k=1}^{j-1}c_k/N^2+O(1/N^3)}\right| \\
        &\leq \left| \sum_{j=1}^{N-1} a_j T_j \lp 1 + \sum_{k=1}^{j-1} \frac{c_k}{N^2} + O\lp\frac{1}{N^3}\rp \rp \right| \\
        &\leq \left|\sum_{j=1}^{N-1}a_j T_j \right| + O\lp\frac{1}{N}\rp.
    \end{align*}
     Consequently, 
     \begin{equation*}
         |q_N| = |q_N-T_N| \leq \left|\sum_{j=1}^{N-1}a_j T_j \right| + O\lp\frac{1}{N}\rp,
     \end{equation*}
    and $|q_N|\leq \frac{C'}{N}$ if $\left|\sum_{j=1}^{N-1}a_j T_j \right| \leq \frac{C}{N}$.

    An analogous argument can be made to show:
    \[
    |q_{N+1}-1|=|q_{N+1}-T_{N+1}|\leq \frac{C'}{N}.\]

Let us work on the general case now where $b_k-a_k=O(1/N^4)$. 
Then $q_j-q_{j-1}=(\rho+a_{j-1})(q_{j-1}-q_{j-2})+(b_{j-1}-a_{j-1})q_{j-2}$. 
Let us write for simplicity $r_{j-1}:=(b_{j-1}-a_{j-1})q_{j-2}$.
From Lemma \ref{lem:q_kdifferencebound} and the conditions on $a_j$ and $b_j$, we have that each $q_j$ is of order at most $N$ and $b_{j-1}-a_{j-1}=O(1/N^4)$. Then each term $|r_j|<C/N^3$ where $C$ is uniform. 
Hence,
\begin{align*}
    q_j-q_{j-1} = \prod_{k=1}^{j-1} (\rho+{a}_{k}) + R_{j-1}
\end{align*}
where $R_{j-1}:= r_{j-1}+r_{j-2}(\rho+a_{j-1})+r_{j-3}(\rho+a_{j-2})(\rho+a_{j-1})+ \cdots+ r_{1}(\rho+a_2)(\rho+a_{3})\cdots (\rho+a_{j-1})$.
Because $|\rho+a_{j-1}|=1+O(1/N^2)$, any product of less than $N$ terms of this form is of order $O(1)$. 
Therefore, for each $1\leq i < j \leq N-1$ we obtain:
\begin{align*}
    |r_{j-i}\prod_{k=1}^{i-1}(\rho+a_{j-k})|<C/N^3,
\end{align*}
which implies that $|R_{j-1}|<C'/N^2$.

Then
    \begin{align*}
        \left|\sum_{j=1}^{N-1} a_j (q_j-q_{j-1})T_{N-j} \right|\leq \left|\sum_{j=1}^{N-1}a_jT_j\right|+\left|\sum_{j=1}^{N-1}a_jR_{j-1}T_{N-j} \right|+O\lp\frac{1}{N}\rp.
\end{align*}
We can check that
\begin{align*}
\left|\sum_{j=1}^{N-1}a_jR_{j-1}T_{N-j} \right| < N\frac{C}{N^2}\frac{C'}{N^2}N = O(1/N^2).
\end{align*}
\end{proof}

In the following Corollaries, we provide examples of sequences $(c_k)$ that satisfy the summation condition \eqref{eq: ak/N^2 sum requirement}. 

\begin{cor}\label{consecutive}
	Let $(c_k)\subset\mathbb{C}$ be a bounded sequence such that $c_k+c_{k+1}=O\lp\frac{1}{N}\rp$ for $k$ odd. Then the coefficients $a_k=\frac{c_k}{N^2}$
satisfy condition \eqref{eq: ak/N^2 sum requirement}.
\end{cor}

\begin{proof}
	Let $(c_k)\subset\mathbb{C}$ be bounded such that $c_k+c_{k+1}=O\lp\frac{1}{N}\rp$ for $k$ odd. 
	Expanding the expression \eqref{eq: ak/N^2 sum requirement} 
	\begin{align*}
        \left|\sum_{j=1}^{N-1}a_jT_j\right| 
        &= \frac{1}{N^2}\left|\sum_{j=1}^{N-1}  c_jT_j\right|=\frac{1}{N^2}\left|\sum_{1\leq j\leq N-1, j \textrm { odd }} (c_jT_j+c_{j+1}T_{j+1})\right|\\
		&= \frac{1}{N^2}\left|\sum_{1\leq j\leq N-1, j \textrm { odd }} (c_j+c_{j+1})T_j+c_{j+1}(T_{j+1}-T_j)\right|\\
        &\leq \frac{1}{N^2}\left(\sum_{1\leq j\leq N-1, j \textrm { odd }} |c_j+c_{j+1}||T_j|+\sum_{j=1}^{N-1}|c_{j+1}||T_{j+1}-T_j|\right)\\
        &\leq \frac{1}{N^2}\left(N\frac{C}{N}N+NC\right)=\frac{2C}{N}.
	\end{align*}
\end{proof}

Similarly we have:

\begin{cor}\label{expsums}
	Let $(c_k)\subset\mathbb{C}$ such that $c_k = Ce^{2\pi ik/N}$ for $C\in\mathbb{C}$. 
	Then the coefficients $a_k=\frac{c_k}{N^2}$ satisfy condition \eqref{eq: ak/N^2 sum requirement}.
\end{cor}

\begin{proof}
	We can rewrite $c_k = Ce^{2\pi ik/N}$ as $c_k=C\rho^k$, since $\rho=e^{2\pi i/N}$. Recall that $T_k=\frac{\rho^k-1}{\rho-1}$.
	Then the expression in \eqref{eq: ak/N^2 sum requirement} becomes 
	\begin{align*}
        \sum_{k=1}^{N-1} a_k T_k
		&= \frac{1}{N^2}\sum_{k=1}^{N-1} C\rho^k\frac{\rho^k-1}{\rho-1}\\
        &= \frac{C}{N^2}\left(\sum_{k=1}^{N-1} \frac{\rho^{2k}}{\rho-1}-\sum_{k=1}^{N-1} \frac{\rho^{k}}{\rho-1}\right)=0.
	\end{align*}
\end{proof}

\subsection{Combined rotation and translation}

\begin{prop}\label{prop:qklimits}
	Let 
    \begin{align*}
		q_{k+1} &= (1+\rho+b_k-\epsilon_k^2) q_{k}-(\rho+b_k) q_{k-1}, & q_0 &=0, & q_1 &=1, &
	\end{align*}
    for $k\geq2$.
    If $\epsilon_k$ and $b_k$ satisfy any of the conditions (i)–(iii), where $c_{k}\in\mathbb{C}$ and $c_{k}$ bounded for all $k\in \mathbb{N}$;
then $q_N \to 0$ and $q_{N+1} \to 1$ as $N \to \infty$.
    	\begin{enumerate}
		\item[i)] $\displaystyle{b_k = \frac{c_k}{N^3}+O\lp\frac{1}{N^4}\rp}$ for bounded $c_k\in\mathbb{C}$ and $\displaystyle{\epsilon_k= O\lp\frac{1}{N^2}\rp}$; or
		\item[ii)]  $\displaystyle{b_k = \frac{c_k}{N^2}+O\lp\frac{1}{N^4}\rp}$ such that $c_k\in\mathbb{C}$ is bounded, $\displaystyle{c_k+c_{k+1}=O\lp\frac{1}{N}\rp}$ for $k$ odd and $\displaystyle{\epsilon_k= O\lp\frac{1}{N^2}\rp}$; or
        \item[iii)] $\displaystyle{b_k = \frac{c_{k}}{N^2}+O\lp\frac{1}{N^{4}}\rp}$ where $\displaystyle{c_{k} = Ce^{2\pi i k/N}}$ for $C\in\mathbb{C}$ and $\displaystyle{\epsilon_k= O\lp\frac{1}{N^2}\rp}$. 
	\end{enumerate}
\end{prop}	

\begin{proof}

\begin{enumerate}
		\item[i)] Using Proposition \ref{prop:cubic} with $b_k = \frac{c_k}{N^3}+O\lp\frac{1}{N^4}\rp$ and         
        \begin{align*}
            a_{k} := b_{k}- \epsilon_{k}^{2} =\frac{c_k}{N^3}+O\lp\frac{1}{N^4}\rp -  O\lp\frac{1}{N^4}\rp = \frac{c_k}{N^3}+O\lp\frac{1}{N^4}\rp,
        \end{align*}
we have $q_N\rightarrow 0 \quad\text{and}\quad q_{N+1}\rightarrow 1.$
    \item[ii)] and iii) the sequences $b_{k}$ and $a_{k}:=b_k-\epsilon_k^2$ are of the form required by Proposition~\ref{pro:summationconstant}, so we need only check that $c_{k}$ satisfy condition ~\eqref{eq: ak/N^2 sum requirement}. This is shown for (ii) in Corollary~\ref{consecutive} and for (iii) in Corollary~\ref{expsums}.

\end{enumerate}
\end{proof}

\subsection{Different initial conditions}		

We now consider the other recursive polynomial in Lemma \ref{lem:recursivepoly}:
\begin{align*}
	r_{k+1} &= (1+\rho + a_{k})r_k -(\rho+b_k)r_{k-1}, & r_0&=1, & r_1&=1. &
\end{align*}
Note that this recursive polynomial has its starting term, $r_0$, initialized at $1$ instead of $0$, so our previous results will not apply. 
However, we can re-write $r_k$ as a difference of recursive polynomials which allows us to adapt our previous results.  
As before, we write our polynomials in a more general form.
\begin{lemma}\label{lem:r_kdiffeq}
	Consider the sequences 
	\begin{align*}
		q_{k+1} &= (1+\rho+a_k) q_k - (\rho+b_k)q_{k-1}, & q_0&=0, & q_1&=1; & \\
		r_{k+1} &= (1+\rho+a_k) r_k - (\rho+b_k)r_{k-1}, & r_0&=1, & r_1&=1; &
	\end{align*}
    for $k\geq 2$. 
	Then 
	\begin{align*}
		r_k &= q_k - (\rho +b_1)s_{k-1}
	\end{align*}
	where the sequence $s_k$ is given by the conditions $s_0=0$, $s_1 = 1$ and for $k\geq 2$ we have    
	\begin{align*}
		s_{k+1} &= (1+\rho+a_{k+1})s_k - (\rho+b_{k+1})s_{k-1}.
	\end{align*}
\end{lemma}

\begin{proof}
	The proof follows immediately by writing down $r_k-q_k$ and checking the initial conditions. 
\end{proof}
 We will show that $r_N\rightarrow 1$ and $r_{N+1}\rightarrow 1$ as $N\rightarrow\infty$ by calculating the limits of $q_k$ and $s_k$. 
We know the limits of $q_k$ by Proposition \ref{prop:qklimits}. 
The sequences $q_{k}$ and $r_{k}$ are fixed by specifying sequences $a_{k}$ and $b_{k}$. This determines a corresponding sequence $s_{k}$, which has the same form as $q_{k}$ but depends on the sequences $a_{k+1}$ and $b_{k+1}$. 
If the sequences $a_{k+1}$ and $b_{k+1}$ satisfy the conditions of Proposition~\ref{prop:qklimits}, then convergence of $s_{k}$ follows. 
We verify this in the proof in the following proposition.

\begin{prop}\label{prop:rklimits}  
	Let 
	\begin{align*}
		r_{k+1} &= (1+\rho+b_k-\epsilon_k^2)r_k - (\rho+b_k)r_{k-1}, & r_0&=1, & r_1 &=1. &
	\end{align*} 
    If $\{b_{k}\}_{k\geq1}$ and $\{\epsilon_{k}\}_{k\geq1}$ satisfy any of the conditions (i)-(iii).  Suppose $c_{k}\in\mathbb{C}$ and $(c_{k})$ bounded for all $k\in \mathbb{N}$.
	\begin{enumerate}
		\item[i)] $\displaystyle{b_k = \frac{2\pi ic_k}{N^3}+O\lp\frac{1}{N^4}\rp}$ and $\displaystyle{\epsilon_k= O\lp\frac{1}{N^2}\rp}$; or
		\item[ii)]  $\displaystyle{b_k = \frac{2\pi ic_k}{N^2}+O\lp\frac{1}{N^3}\rp}$ such that $\displaystyle{c_k+c_{k+1}=O\lp\frac{1}{N}\rp}$ for $k$ odd and $\displaystyle{\epsilon_k= O\lp\frac{1}{N^2}\rp}$; or 
		\item[iii)] $\displaystyle{b_k = \frac{c_{k}}{N^2}+O\lp\frac{1}{N^{4}}\rp}$ where $\displaystyle{c_{k} = Ce^{2\pi i k/N}}$ for $C\in\mathbb{C}$ and $\displaystyle{\epsilon_k= O\lp\frac{1}{N^2}\rp}$
	\end{enumerate}
	Then as $N\rightarrow\infty$,
	\begin{equation*}
		r_N\rightarrow 1\quad\text{and}\quad r_{N+1}\rightarrow 1.
	\end{equation*}
\end{prop}

\begin{proof}
    By Lemma \ref{lem:r_kdiffeq},
    \[r_k=q_k-(\rho+b_1)s_{k-1},\]
    from which we will calculate $r_N\rightarrow1$ and $r_{N+1}\rightarrow1$ as $N\rightarrow\infty$.

    First, recall that in Propositions~\ref{prop:cubic} and \ref{pro:summationconstant} we showed
    \[q_N\rightarrow0\quad \text{and}\quad q_{N+1}\rightarrow1.\]

    Next, observe that $s_k$ is $q_k$ but with shifted coefficients i.e. the sequence  $\{b_k\}_{k=1}^{N}$ satisfies the conditions of Proposition \ref{prop:qklimits}; hence, the arguments for $q_k$ in Propositions \ref{prop:cubic} and \ref{pro:summationconstant} apply for $s_k$ from which we get
    \begin{align*}
      |s_{N-1}+\rho^{-1}| &= |s_{N-1}-T_{N-1}| \leq O\left(\frac{1}{N}\right)\\
       |s_{N}| &= |s_{N}-T_{N}| \leq O\left(\frac{1}{N}\right)
    \end{align*}
    where the $N-1$ case follows similarly to the $N$ and $N+1$ cases; in case (iii), we use Corollary \ref{expsums} with $C=\rho$.
    
    This allows us to conclude
    \begin{equation*}
		r_N = q_N - (\rho+b_1)s_{N-1} \rightarrow 0 - \left(\rho+0\right)\lp -\rho^{-1}\rp = 1
	\end{equation*}
	and 
	\begin{equation*}
		r_{N+1} = q_{N+1} - (\rho+b_1)s_{N} \rightarrow 1 -(\rho+0)(0) = 1,
	\end{equation*}
	as $N\rightarrow\infty$. 
\end{proof}

\section{Non-Autonomous Implosions}\label{sec:nonautonomousimplosions}

In this section we prove all our theorems related to non-autonomous parabolic implosion. We use the estimates obtained in Section \ref{sec:orthogonalpolynomials}, as well as the results explained in Section \ref{sec:compositionoflineartransformations}, that link coefficients of compositions of M\"obius transformations.

We will separate our theorems into two categories. As before, starting from the map $\displaystyle{f(z)= \frac{z}{1-z}}$, we first investigate multiplicative perturbations of $f$. On our second set of theorems we focus on perturbations that are of both multiplicative and additive form.
\setcounter{mainthm}{0}
	\begin{mainthm}
		Let \[f_k(z):= \frac{\rho_kz}{1-z},\]
		where $\rho_k = e^{2\pi i\theta_k}$ and $\theta_k$ satisfies any of the following conditions:
		
		\begin{enumerate}
		\item
		$\displaystyle{\theta_k = \frac{1}{N} + O\lp\frac{1}{N^3}\rp}$; or
		\item
		$\displaystyle{\theta_k = \frac{1}{N} + \frac{c_k}{N^2}}$ where $c_k+ c_{k+1} = O\lp \frac{1}{N} \rp$ for $k$ odd; or
		\item
		$\displaystyle{\theta_k = \frac{1}{N} + \frac{c_k}{N^2}}$ where $\displaystyle{c_k =Ce^{2\pi ik/N}} $  for $C\in\mathbb{C}$.
		\end{enumerate}
		Then:
		\begin{equation*}
			f_N\circ f_{N-1} \circ \cdots \circ f_1 \rightarrow \text{Id}
		\end{equation*}
		as $N\to \infty$ uniformly on compact sets of $\mathbb{C}$.
	\end{mainthm}
	
	\begin{proof}
	By Corollary \ref{cor:onlyrho}, 
		\[F_{N}\coloneqq f_N\circ f_{N-1} \circ \cdots \circ f_1(z) = \frac{A_N z }{C_N z + 1}\] where $A_N = q_{N+1}-q_N$ and  $C_N=-q_N$ where:
\begin{align*}
q_{k+1} &= (1+\rho_k) q_k-\rho_k q_{k-1}, & q_0&=0, & q_1&=1.&
\end{align*}		
We will show that given the $\theta_{k}$ above, the coefficients in the orthogonal polynomial $q_k$ satisfy Proposition \ref{prop:qklimits} yielding 
\begin{align*}
 |A_N| &= |q_{N+1}-q_N| \rightarrow 1 \ \text{as} \ N\rightarrow \infty, \text{ and} \\
 |C_N| &= |q_N| \rightarrow 0 \ \text{as} \ N\rightarrow \infty. 
\end{align*}  
Thus, $F_{N}\rightarrow \text{Id}$ as $N\rightarrow \infty.$

We have $\epsilon_{k} \equiv 0$ and $b_{k}\coloneqq  \rho_{k}-\rho$. Therefore, the above sequence $q_{k}$ becomes
\begin{align*}
    q_{k+1} &= (1+\rho +b_{k}-\epsilon_{k}^2) q_k-(\rho+b_k)q_{k-1}, & q_0&=0, & q_1&=1.&
\end{align*}
Therefore, we are in the setting of Proposition \ref{prop:qklimits}.  

The condition $\epsilon_{k} = O\lp\frac{1}{N^2}\rp$ of Proposition \ref{prop:qklimits}, is trivially satisfied, so we need only show the $b_{k}$ conditions are satisfied. 
For $\theta_k = \frac{1}{N} + O\lp\frac{1}{N^3}\rp$, we have
\begin{enumerate}
		\item If $\displaystyle{\theta_k = \frac{1}{N} + O\lp\frac{1}{N^3}\rp}$, then
        			\begin{align*}
        				b_{k} = \rho_{k} - \rho &= e^{2 \pi i\lp \frac{1}{N} +O\lp\frac{1}{N^3}\rp \rp} - e^{\frac{2 \pi i}{N}}\\
        				& = e^{\frac{2 \pi i}{N}} \lp e^{2 \pi i O\lp\frac{1}{N^3} \rp} -1 \rp \\
        				& = O\lp\frac{1}{N^3}\rp.
        			\end{align*}
		This satisfies condition (i) Proposition \ref{prop:qklimits}.
		\item If $\displaystyle{\theta_k = \frac{1}{N} + \frac{c_k}{N^2}}$ and $c_k+ c_{k+1} = O\lp \frac{1}{N} \rp$, then
		\begin{align*}
b_{k} = \rho_{k} - \rho &= e^{2 \pi i\lp \frac{1}{N} + \frac{c_k}{N^2} \rp} - e^{\frac{2 \pi i}{N}}\\
        				& = e^{\frac{2 \pi i}{N}} \lp e^{2 \pi i \frac{c_k}{N^2} } -1 \rp \\
        				& = \frac{2\pi i c_{k}}{N^2} + O\left(\frac{1}{N^{3}}\right),
        			\end{align*}
			satisfying condition (ii) of Proposition \ref{prop:qklimits}.
		\item If $\displaystyle{\theta_k = \frac{1}{N} + \frac{c_k}{N^2}}$ and $\displaystyle{c_k = Ce^{2\pi ik/N}}$, then
\begin{align*}
b_{k} = \rho_{k} - \rho &= e^{2 \pi i\lp \frac{1}{N} + \frac{C}{N^2}e^{2\pi ik/N} \rp} - e^{\frac{2 \pi i}{N}}\\
        				& = e^{\frac{2 \pi i}{N}} \lp e^{\frac{2\pi iC }{N^2}e^{2\pi ik/N}  } -1 \rp \\
        				& = \frac{2\pi iC}{N^2}e^{2\pi ik/N} + O\left(\frac{1}{N^{4}}\right),
    \end{align*}        
        satisfying condition (iii) of Proposition \ref{prop:qklimits}.
\end{enumerate} 
	\end{proof}

	\begin{mainthm} 
		Let
		\begin{equation*}
			f_k(z) := \frac{\rho_k z}{1-z} + \epsilon_k^2,
		\end{equation*}
		where $\rho_k=e^{2\pi i \theta_k}.$ Suppose $\theta_k$ and $\epsilon_k$ satisfy any of the following conditions:
		\begin{enumerate}
		\item
		$\displaystyle{\theta_k = \frac{1}{N} + O\lp\frac{1}{N^3}\rp}$ and $\displaystyle{\epsilon_k=  O\lp \frac{1}{N^2}\rp}$. 
		\item
		$\displaystyle{\theta_k = \frac{1}{N} + \frac{c_k}{N^2}}$ where $\displaystyle{c_k+c_{k+1}=O\lp \frac{1}{N}\rp}$ for odd $k$, and $\displaystyle{\epsilon_k=  O\lp \frac{1}{N^2}\rp}$. \\
        \item
  		$\displaystyle{\theta_k = \frac{1}{N} + \frac{c_k}{N^2}}$ where $\displaystyle{c_k = Ce^{2\pi ik/N},}$ and $\displaystyle{\epsilon_k=  O\lp \frac{1}{N^2}\rp}$. \\
		\item
		$\displaystyle{\theta_k = 0}$, and $\displaystyle{\epsilon_k=\frac{\pi}{N}+ O\lp \frac{1}{N^3}\rp}$. 
		\item
		$\displaystyle{\theta_k = 0}$, and $\displaystyle{\epsilon_k=\frac{\pi}{N}+  \frac{c_k}{N^2}}$ where $\displaystyle{c_k+c_{N-k}=O\lp \frac{1}{N}\rp}$. 
		\end{enumerate}
		Then 
		\begin{equation*}
			f_N\circ f_{N-1} \circ \cdots \circ f_1 \rightarrow \text{Id}
		\end{equation*}
        as $N\rightarrow\infty$, uniformly on compact subsets of $\mathbb{C}$.
	\end{mainthm}

\begin{proof}
    When $\epsilon_{k} \neq 0$, then by Lemma \ref{lem:recursivepoly}, 
	\begin{align*}
	F_{N}\coloneqq f_N\circ f_{N-1} \circ \cdots \circ f_1(z) = \frac{A_N z +B_N }{C_N z + D_N}
	\end{align*}
where $A_N = q_{N+1}-q_N$, $B_N = r_{N}-r_{N+1}$ $C_N=-q_N$, and $D_N = r_N$ where
\begin{align*}
q_{k+1} &= (1+\rho_k-\epsilon_k^2) q_k-\rho_k q_{k-1}, & q_0&=0, & q_1&=1;& \\
r_{k+1} &= (1+\rho_k-\epsilon_k^2) r_k-\rho_k r_{k-1}, & r_0&=1, & r_1&=1.&
\end{align*}	
Again, since $b_k:= \rho_k-\rho$ we are in the setting of Proposition \ref{prop:qklimits} and Proposition \ref{prop:qklimits} and \ref{prop:rklimits}. 
For the $\theta_{k}$ in cases (1) and (2), and (3) we computed $b_{k}$ in the proof of Theorem \ref{thm:A}: 
\begin{enumerate}
   \item $\displaystyle{b_{k}=O\lp\frac{1}{N^{3}}\rp}$,
   \item $\displaystyle{b_{k} = \frac{2\pi i c_{k}}{N^2} + O\left(\frac{1}{N^{3}}\right)}$ where $\displaystyle{c_k+c_{k+1}=O\lp\frac{1}{N}\rp}$, and
   \item $\displaystyle{b_k =\frac{2\pi iC}{N^2}e^{2\pi ik/N} + O\left(\frac{1}{N^{4}}\right)}$. 
\end{enumerate}
Since $\epsilon_{k} =O\lp \frac{1}{N^2}\rp$ and each $b_k$ satisfies the corresponding condition (i)-(iii) of Proposition \ref{prop:qklimits}, the terms $q_N \to 0$ and $q_{N+1} \to 1$ as $N \to \infty$. 
The assumptions on $b_{k}$ and $\epsilon_{k}$ are identical for Proposition \ref{prop:rklimits}. Therefore, it also follows that $r_N\rightarrow 1$ and $r_{N+1}\rightarrow 1$ as $N \to \infty$.
Combining the results from Propositions \ref{prop:qklimits} and \ref{prop:rklimits}, we can conclude that
\begin{align*}
 |A_N| &= |q_{N+1}-q_N| \rightarrow 1 \\ 
 |C_N| &= |q_N| \rightarrow 0 \\ 
 |B_N| &= |r_{N}-r_{N+1}| \rightarrow 0 \\ 
 |D_N| &= |r_N| \rightarrow 1 \\ 
\end{align*}  
as $N\rightarrow\infty$. 
The above together again imply $F_{N}\rightarrow \text{Id}$ as $N\rightarrow \infty.$ 

Cases (4) and (5) were computed in \cite{Viv20} and are included for completeness.
\end{proof}

\section{Difference with the Additive Condition}\label{sec:differencewiththeadditive}

In the previous sections we analyzed the multiplicative setting, where compositions of maps of the form 
\[
f_k(z) = \frac{\rho_k z}{1-z}
\]
exhibit delicate convergence properties depending on the sequence $\{\rho_k\}$. 
We now contrast this with the additive case, which at first sight appears closely related but leads to very different asymptotic behavior.

Recall the following family of maps:
\begin{align*}
g_k(z)=\frac{z}{1-z}+\epsilon_k^2
\end{align*}
where we fix $N$.
Maps of this type were considered in \cite{Viv20} where it is proven that:
\[
g_N \circ g_{N-1} \circ \cdots \circ g_1 \to \mathrm{Id},
\]
under specific conditions on the sequence $\{\epsilon_k\}$.

Each individual $g_k$ is analytically conjugated to a map $f_k$ of the form
\begin{align*}
f_k(z) = \frac{\rho_k z}{1-z},
\end{align*}
provided that the parameters satisfy
\begin{equation*}
\sqrt{\rho_k} + \frac{1}{\sqrt{\rho_k}} = 2 - \epsilon_k^2.
\end{equation*}

However, the following theorem shows that, even when the maps are pairwise conjugated, the convergence of their compositions may behave very differently.

\begin{mainthm}

There exist sequences $\{\epsilon_k\}$ and $\{\rho_k\}$ for $1\leq k\leq N$, such that:
\begin{enumerate}
\item For each $k$, the maps $\displaystyle{f_k(z)=\frac{\rho_k z}{1-z}}$ and $\displaystyle{g_k(z)=\frac{z}{1-z}+\epsilon_k^2}$ are conjugated.\label{part1:thmC}
\item We have $g_N \circ g_{N-1} \circ\cdots \circ g_1$ converges uniformly on compacts to the Identity.\label{part2:thmC}
\item The composition $f_N \circ f_{N-1}\circ\cdots \circ f_1(z)$ does not converge pointwise to $z$ for any $z \neq 0$.\label{part3:thmC}
\end{enumerate}
\end{mainthm}

This example illustrates that conjugacy at the level of individual maps does not imply comparable asymptotic behavior under composition.

\begin{proof}
As in the one variable situation, the relationship for the conjugacy is satisfied when:
\[
\rho_k=\exp(2i\theta_k)\textrm{ and }\epsilon_k = 2\sin(\theta_k/2).
\]
Assume for simplicity that $N=2m$. We choose: $\theta_k= \frac{\pi}{N-1}$ for $1\leq k \leq m $ and $\theta_k= \frac{\pi}{N+1}$ for $m+1 \leq k\leq 2m=N$.

In \cite{Viv20}, it is proven that: if $\epsilon_k = \frac{\pi}{N} + \frac{\alpha_k}{N^2}$ such that $\alpha_k + \alpha_{N-k} = O\lp 1/N\rp$, then \eqref{part2:thmC} holds. It is clear that this condition is satisfied for the choice of $\epsilon_k$.
In fact, the choice of $\epsilon_k$ implies that
$\alpha_k= 
\begin{cases}
  \pi + O(1/N) & \text{if $1\leq k \leq m$} \\
  -\pi + O(1/N) & \text{if $m+1\leq k \leq 2m.$}
\end{cases}
$
Therefore $\alpha_k+\alpha_{N-k} = O(1/N)$ and \eqref{part2:thmC} holds.

We now prove \eqref{part3:thmC}. For simplicity let us rename $s=\exp(2i\pi/(N-1))$ and $t=\exp(2i\pi/(N+1))$. We use the formula $q_{k+1}=(1+\rho_k)q_k -\rho_kq_{k-1}$ for $q_N$ and see that
        \begin{align*}
			q_N&= 1+s+s^2+\ldots+s^{m-1}+s^m\left(1+t+t^2+\ldots +t^{m-1}\right)\\
			&= \frac{s^m-1}{s-1}+ s^m\frac{t^m-1}{t-1}.
		\end{align*}
	For the first term, we have the following:
    \begin{align*}
    \frac{s^m-1}{s-1}=\frac{e^{iN\pi/(N-1)}-1}{e^{2i\pi/(N-1)}-1} = \frac{i}{\pi}(N-1) + \frac{1}{2} + O(1/N).
	\end{align*}
    Likewise:
    \begin{align*}
    \frac{t^m-1}{t-1}=\frac{e^{iN\pi/(N+1)}-1}{e^{2i\pi/(N+1)}-1} =\frac{i}{\pi}(N+1) +\frac{3}{2} + O(1/N).
	\end{align*}
    Putting it all together, along with the fact that $s^m = -1 - \frac{i\pi}{N}+O(1/N^2)$:
    \begin{align*}
    q_N =  -\frac{2i}{\pi} + O(1/N).
	\end{align*}
	Consequently,	
	\begin{align*}
		|q_N-T_N| &\geq \frac{1}{\pi}.
	\end{align*}
    for $N$ large enough.
	Thus, we do not get convergence to the identity.
\end{proof}

\section{Bifurcations for Skew-Products}\label{sec:bifurcationsforskew-products}
A standard way in which non-autonomous parabolic phenomena arise in higher dimension is through \emph{skew-product} maps
\[
F(z,w)=(f_w(z),g(w)),
\]
where the base dynamics $w\mapsto g(w)$ generates a parameter sequence
\(
w_0,w_1,\dots,w_{N-1}
\)
and hence a non-autonomous composition in the fiber,
\[
f_{w_{N-1}}\circ \cdots \circ f_{w_1}\circ f_{w_0}.
\]

In the parabolic context, one is interested in regimes where the fiber maps are small perturbations of the model parabolic M\"obius map
\(
z\mapsto \frac{z}{1-z}
\)
and where the induced non-autonomous composition exhibits a \emph{bifurcation-like} effect as the base point $w_0=w_0(N)\to 0$ and the number of iterates $N\to\infty$.

The results of Section \ref{sec:nonautonomousimplosions} provide a flexible mechanism for producing such examples: if along the base orbit one can write the fiber maps in the form
\[
f_k(z)=\rho_k\,\frac{z}{1-z}+\epsilon_k^2,
\qquad \rho_k=e^{2\pi i\theta_k},
\]
with $\theta_k=\frac{1}{N}+O(N^{-2})$ (or $\theta_k=0$ in the additive scaling regime) and $\epsilon_k=O(N^{-2})$ in one of the admissible patterns from Theorems~\ref{thm:A}--\ref{thm:B}, then the induced $N$--step fiber composition converges to the identity. 
In particular, in each of the examples below, we choose a base map $g$ and an initial parameter $w_0=w_0(N)$ so that the resulting sequence $\{\theta_k,\epsilon_k\}$ falls under one of our theorems, yielding
\[
F^{N}(z,w_0)\longrightarrow (z,0)
 \text{ as }N\to\infty,
\]
uniformly on compact subsets in the $z$--variable.

\begin{exmp}[Constant base parameter]
Consider
\[
F(z,w)=\left(e^{2\pi i w}\frac{z}{1-z},\,w\right).
\]
Then $F^{N}(z,w_0)\to (z,0)$ for $w_0=\frac{1}{N}$.
\end{exmp}

\begin{proof}
If $w_0=\frac{1}{N}$, then the base coordinate is constant: $w_k=w_0$ for all $k$. Hence the fiber map at each step is
\[
f(z)=e^{2\pi i/N}\frac{z}{1-z}.
\]
Therefore
\[
F^N\!\left(z,\frac{1}{N}\right)=\left(f^{\circ N}(z),\frac{1}{N}\right)=\left(z,\frac{1}{N}\right)
\]
by Example~\ref{exmp: basic}, and the claim follows.
\end{proof}

\begin{exmp}[Alternating multiplicative perturbations]\label{exmp:skew-alternating}
Consider
\[
H(z,w)=\left(e^{2\pi i(\frac{1}{N}+w)}\frac{z}{1-z},\, -w\right).
\]
Then $H^{N}(z,w_0)\to (z,0)$ for $w_0=-\frac{1}{N^2}$.
\end{exmp}

\begin{proof}
Set $w_0=-\frac{1}{N^2}$ and define $c_k\in\{\pm1\}$ by
\[
c_k=\begin{cases}
1,& k\ \text{odd},\\
-1,& k\ \text{even}.
\end{cases}
\]
Then $w_k=c_k w_0$ and the $k$th fiber map is
\[
f_k(z)=e^{2\pi i\theta_k}\frac{z}{1-z},
\qquad \theta_k=\frac{1}{N}+w_k=\frac{1}{N}+\frac{c_k}{N^2}.
\]
Since $c_k+c_{k+1}=0$ for all $k$, Theorem~\ref{thm:A}(2) applies and gives
\(
f_N\circ\cdots\circ f_1\to \mathrm{Id}.
\)
Also $w_N=(-1)^N w_0\to 0$. Hence $H^N(z,w_0)\to (z,0)$.
\end{proof}

\begin{exmp}[Rotating multiplicative perturbations]
Consider
\[
G(z,w)=\left(e^{2\pi i(\frac{1}{N}+w)}\frac{z}{1-z},\, e^{2\pi i/N}\,w\right).
\]
Then $G^{N}(z,w_0)\to (z,0)$ for $w_0=\frac{e^{2\pi i/N}}{N^2}$.
\end{exmp}

\begin{proof}
With $w_0=\frac{e^{2\pi i/N}}{N^2}$ we have
\(
w_k=\frac{e^{2\pi i(k+1)/N}}{N^2}.
\)
Define
\[
\theta_k=\frac{1}{N}+\frac{e^{2\pi i k/N}}{N^2},
\qquad 
f_k(z)=e^{2\pi i\theta_k}\frac{z}{1-z}.
\]
Then $\theta_k=\frac{1}{N}+w_{k-1}$, so the $k$th fiber map along the orbit is precisely $f_k$. By Theorem~\ref{thm:A}(3),
\(
f_N\circ\cdots\circ f_1\to \mathrm{Id}.
\)
Moreover $|w_N|=1/N^2\to 0$, so $G^{N}(z,w_0)\to (z,0)$.
\end{proof}

\begin{exmp}[Including additive perturbations: constant base]
Consider
\[
F_1(z,w)=\left(e^{2\pi i w}\frac{z}{1-z}+w^4,\,w\right).
\]
Then $F_1^{N}(z,w_0)\to (z,0)$ for $w_0=\frac{1}{N}$.
\end{exmp}

\begin{proof}
If $w_0=\frac{1}{N}$, then $w_k=w_0$ for all $k$, and each fiber map equals
\[
f(z)=\rho\,\frac{z}{1-z}+\epsilon^2,
\qquad 
\rho=e^{2\pi i/N},
\qquad 
\epsilon^2=w_0^4=\frac{1}{N^4}.
\]
Thus $\theta=\frac{1}{N}$ and $\epsilon=O(N^{-2})$, so Theorem~\ref{thm:B}(1) gives
\(
f^{\circ N}\to \mathrm{Id}.
\)
Hence
\(
F_1^N(z,w_0)=\bigl(f^{\circ N}(z),w_0\bigr)\to (z,0).
\)
\end{proof}

\begin{exmp}[Including additive perturbations: alternating base]
Consider
\[
G_1(z,w)=\left(e^{2\pi i(\frac{1}{N}+w)}\frac{z}{1-z}+w^2,\,-w\right).
\]
Then $G_1^{N}(z,w_0)\to (z,0)$ for $w_0=-\frac{1}{N^2}$.
\end{exmp}

\begin{proof}
Let $w_0=-\frac{1}{N^2}$ and define $c_k$ as in Example~\ref{exmp:skew-alternating}, so $w_k=c_k w_0$ and $w_k^2\equiv \frac{1}{N^4}$. The induced fiber maps are
\[
f_k(z)=e^{2\pi i\theta_k}\frac{z}{1-z}+\epsilon_k^2,
\qquad 
\theta_k=\frac{1}{N}+w_k=\frac{1}{N}+\frac{c_k}{N^2},
\qquad 
\epsilon_k^2=w_k^2=\frac{1}{N^4}.
\]
Since $c_k+c_{k+1}=0$ and $\epsilon_k=O(N^{-2})$, Theorem~\ref{thm:B}(2) applies and yields
\(
f_N\circ\cdots\circ f_1\to \mathrm{Id}.
\)
Also $w_N=(-1)^N w_0\to 0$, hence $G_1^{N}(z,w_0)\to (z,0)$.
\end{proof}

The examples above illustrate how the abstract non-autonomous convergence results proved in this paper naturally arise in the study of skew-product dynamics.  
In each case, the dynamics of the base variable $w$ produces a controlled drift through a parabolic regime in the fiber, leading to a non-autonomous composition whose asymptotics are governed by Theorems~\ref{thm:A} and~\ref{thm:B}.  

While these examples do not by themselves yield global dynamical consequences such as the existence of wandering Fatou components, they capture the local bifurcation mechanisms that underpin such constructions in higher dimensional complex dynamics.

This perspective complements the work of Astorg and collaborators \cite{ABD+16, AstBoc22}, where similar skew-product configurations arise implicitly, and suggests that orthogonal polynomial techniques may offer a useful analytic framework for studying more intricate skew-product bifurcations, including random or weakly correlated perturbative regimes.  

We expect that these methods can be further adapted to settings where the base dynamics is no longer periodic or deterministic, potentially leading to new examples of non-autonomous parabolic phenomena in several complex variables.

\section{Random Compositions}\label{sec:randomcompositions}

In this section, we return to the purely additive setting studied by Vivas \cite{Viv20}, where the fiber maps are given by
\[
   f_k(z) \coloneqq \frac{z}{1-z} + \epsilon_k^2,
\]
but we now allow the perturbation parameters $\{\epsilon_k\}_{k\geq 1}$ to be \emph{random}.  
Our goal is to show that, while deterministic perturbations at the critical scale $1/N^2$ require delicate arithmetic cancellations, introducing randomness allows one to recover convergence under significantly weaker assumptions.

Throughout this section we restrict to the autonomous multiplicative case $\rho_k\equiv 1$, so that the dynamics is governed entirely by additive perturbations.  
In this setting the $N$th composition
\[
   F_N := f_N \circ f_{N-1} \circ \cdots \circ f_1
\]
can be expressed as a M\"obius transformation whose coefficients are determined by two recursive sequences $\{q_k\}$ and $\{r_k\}$.  
This representation, introduced in \cite{Viv20}, allows us to reduce the problem of convergence of $F_N$ to quantitative estimates on these sequences.

When $\rho_{k}\equiv 1$, we have $\rho_k=\rho+b_k\equiv 1$ for all $k$, and hence
$\rho+a_k = 1-\epsilon_k^2$.  
Specializing Lemma~\ref{lem:recursivepoly} to this case recovers Lemma~1 of \cite{Viv20}, which shows that
\[
   F_N := f_N \circ f_{N-1} \circ \cdots \circ f_1
\]
is a M\"obius transformation with coefficients determined by the sequences $\{q_k\}$ and $\{r_k\}$ defined by
\begin{align*}
q_{k+1} &= (2-\epsilon_{k}^{2}) q_k-q_{k-1}, & q_0&=0, & q_1&=1, \\
r_{k+1} &= (2-\epsilon_{k}^{2}) r_k-r_{k-1}, & r_0&=1, & r_1&=1.  
\end{align*}

The relationship between $\{q_k\}$ and $\{r_k\}$ is given by Lemma~\ref{lem:r_kdiffeq}.  
In particular, when $\rho_k\equiv 1$, this reduces to Lemma~6 of \cite{Viv20}.

\begin{lemma}\label{lemma:skEpsilon}
If $\rho_k\equiv 1$, then
\[
   r_{k+1} = q_k - s_{k-1},
\]
where the auxiliary sequence $\{s_k\}$ is defined by $s_0=0$, $s_1=1$, and
\[
   s_{k+1}=(2-\epsilon_{k+1}^2)s_k-s_{k-1}.
\]
\end{lemma}

The key idea is to compare the random recursion satisfied by $q_k$ with the deterministic Chebyshev recursion, and to show that random fluctuations remain sufficiently small with overwhelming probability.

We compare the sequence $\{q_k\}$ with the Chebyshev polynomials of the second kind $\{U_k\}$.

\begin{lemma}[{\cite[Lemma~4]{Viv20}}]
Let
\begin{align*}
q_{k+1}&=(x+d_k)q_k-q_{k-1}, & q_0&=0, & q_1&=1,\\
U_{k+1}&=xU_k-U_{k-1}, & U_0&=0, & U_1&=1.
\end{align*}
Let $x=2\cos\theta$.  
If there exist $\epsilon>0$ and $m\in\mathbb{N}$ such that
\[
\sum_{j=1}^{m-1}|d_j q_j|\le \epsilon\sin\theta,
\]
then $|q_n-U_n|\le \epsilon$ for all $1\le n\le m$.
\end{lemma}

The proof begins with the identity
\[
\sin\theta\,(q_n-U_n)=-\textrm{Im}\left(\delta_n e^{-in\theta}\right),
\qquad
\delta_n:=\sum_{j=1}^{n-1} d_j q_j e^{ij\theta}.
\]

We now assume that the coefficients $d_k$ are random.  
Throughout this section we fix $\delta>0$ and write
\begin{align}\label{randomerror}
d_k=\frac{\xi_k}{N^{2+\delta}},  
\end{align}
where $\xi_k$ are any independent, bounded, mean-zero random variables.

\begin{lemma}\label{lem:martingale}
Let
\[
\delta_n:=\sum_{k=0}^{n-1} d_k\,q_k\,e^{ik\theta}, \qquad n\ge1.
\]
Then $\{\delta_n\}_{n\ge1}$ is a martingale with respect to the filtration
$\mathcal{F}_n=\sigma(\xi_0,\dots,\xi_{n-1})$.
\end{lemma}

\begin{proof}
Since $q_k$ depends only on $\xi_0,\dots,\xi_{k-1}$, the variable $\delta_n$ is $\mathcal{F}_n$-measurable.
Moreover,
\[
\mathbb{E}|\delta_n|
\le \sum_{k=0}^{n-1}\mathbb{E}(|d_k||q_k|)<\infty,
\]
because the $\xi_k$ are bounded.

Finally,
\[
\delta_n=\delta_{n-1}+d_{n-1}q_{n-1}e^{i(n-1)\theta}.
\]
Conditioning on $\mathcal{F}_{n-1}$ and using $\mathbb{E}(\xi_{n-1})=0$ gives
\[
\mathbb{E}(\delta_n\mid\mathcal{F}_{n-1})=\delta_{n-1},
\]
so $\{\delta_n\}$ is a martingale.
\end{proof}

Recall from Lemma 5 of \cite{Viv20} that the condition on $d_k$ immediately gives us the bound $q_k$ bounded by $N$ for all $1\leq k \leq N$ with probability one. Therefore we have the following:

\begin{lemma}\label{lem:applyAzuma}
Assume $|\xi_k|\le M$ almost surely and $|q_k|\le N$ with probability one, for $k\le n$.  
Then for any $\lambda_n>0$,
\[
\mathbb{P}(|\delta_n|\ge\lambda_n)
\le
\exp\!\left(-\frac{\lambda_n^2 N^{2(1+\delta)}}{2M^2 n}\right).
\]
Consequently,
\[
\mathbb{P}\bigl(|q_n-U_n|\ge \tfrac{\lambda_n}{\sin\theta}\bigr)
\le
\exp\!\left(-\frac{\lambda_n^2 N^{2(1+\delta)}}{2M^2 n}\right).
\]
\end{lemma}

\begin{proof}
Since
\[
|\delta_n-\delta_{n-1}|
=|d_{n-1}q_{n-1}|
\le \frac{M}{N^{1+\delta}},
\]
Azuma's inequality (see for instance \cite{AlonSpencer}) yields the stated bound.  
The estimate for $|q_n-U_n|$ follows from the identity above.
\end{proof}

\begin{prop}\label{prop:randombigO}
Let $\theta=\pi/N$, $x=2\cos\theta$.  
With probability at least $1-\exp(-N^\delta)$,
\begin{align*}
|q_{N-1}-1| &= O\!\left(N^{-\frac12(1+\delta)}\right),\\
|q_{N}| &= O\!\left(N^{-\frac12(1+\delta)}\right),\\
|q_{N+1}+1| &= O\!\left(N^{-\frac12(1+\delta)}\right),\\
|s_{N-2}-2\cos\theta| &= O\!\left(N^{-\frac12(1+\delta)}\right),\\
|s_{N-1}-1| &= O\!\left(N^{-\frac12(1+\delta)}\right).
\end{align*}
\end{prop}

\begin{proof}
Let $\theta=\pi/N$ and $x=2\cos\theta$.  We apply Lemma~\ref{lem:applyAzuma} with the choice
\[
\lambda_n := \sqrt{2n}\,N^{-(3/2+\delta/2)} .
\]
Then for each $1\le n\le N+1$, Lemma~\ref{lem:applyAzuma} gives
\[
\mathbb{P}\bigl(|\delta_n|\ge \lambda_n\bigr)
\le
\exp\!\left(-c\,\frac{\lambda_n^2 N^{2(1+\delta)}}{n}\right)
=
\exp\!\left(-c\,N^{\delta}\right),
\]
for some constant $c>0$ depending only on the uniform bound on $|\xi_k|$.

Using the identity $\sin\theta\,(q_n-U_n)=-\textrm{Im}(\delta_n e^{-in\theta})$, we obtain
\[
|q_n-U_n|\le \frac{|\delta_n|}{\sin\theta}.
\]
Since $\sin(\pi/N)\sim 1/N$, there is a constant $C>0$ such that
\[
|q_n-U_n|\le C N\,|\delta_n|.
\]
Therefore, on the event $\{|\delta_n|\le \lambda_n\}$ we have
\[
|q_n-U_n|\le C N \lambda_n
\le C \sqrt{2n}\,N^{-(1/2+\delta/2)}
=O\!\left(N^{-\frac12(1+\delta)}\right),
\qquad 1\le n\le N+1.
\]

By a union bound over $1\le n\le N+1$,
\[
\mathbb{P}\Bigl(\max_{1\le n\le N+1}|\delta_n|>\lambda_n\Bigr)
\le (N+1)\exp(-cN^\delta)
\le \exp(-c' N^\delta)
\]
for some $c'>0$ and all $N$ large. Hence, with probability at least
$1-\exp(-c' N^\delta)$ we have simultaneously for all $1\le n\le N+1$,
\[
|q_n-U_n| = O\!\left(N^{-\frac12(1+\delta)}\right).
\]

Now use the explicit values of $U_n$ at $\theta=\pi/N$:
\[
U_{N-1}(2\cos(\pi/N))=\frac{\sin((N-1)\pi/N)}{\sin(\pi/N)}=1,\quad
U_N(2\cos(\pi/N))=\frac{\sin(\pi)}{\sin(\pi/N)}=0,
\]
\[
U_{N+1}(2\cos(\pi/N))=\frac{\sin((N+1)\pi/N)}{\sin(\pi/N)}=-1.
\]
This yields the first three estimates:
\[
|q_{N-1}-1|=O\!\left(N^{-\frac12(1+\delta)}\right),\quad
|q_N|=O\!\left(N^{-\frac12(1+\delta)}\right),\quad
|q_{N+1}+1|=O\!\left(N^{-\frac12(1+\delta)}\right).
\]

Finally, the same argument applies to the shifted sequence $s_n$ (as noted in the remark preceding Lemma~\ref{lem:martingale}), giving
\[
|s_n-U_n|=O\!\left(N^{-\frac12(1+\delta)}\right)
\]
uniformly for $n\le N+1$ on the same high-probability event. In particular,
\[
|s_{N-2}-U_{N-2}|=O\!\left(N^{-\frac12(1+\delta)}\right),\qquad
|s_{N-1}-U_{N-1}|=O\!\left(N^{-\frac12(1+\delta)}\right).
\]
Since
\[
U_{N-2}(2\cos(\pi/N))
=\frac{\sin((N-2)\pi/N)}{\sin(\pi/N)}
=\frac{\sin(2\pi/N)}{\sin(\pi/N)}=2\cos(\pi/N),
\]
and $U_{N-1}(2\cos(\pi/N))=1$, we obtain
\[
|s_{N-2}-2\cos(\theta)|=O\!\left(N^{-\frac12(1+\delta)}\right),\qquad
|s_{N-1}-1|=O\!\left(N^{-\frac12(1+\delta)}\right),
\]
as claimed.
\end{proof}

We are ready now to prove our main theorem in this section.

\begin{mainthm}

For fixed \(N\), let \(\{\eta_{k}\}_{k=1}^{N+1}\) be independent, bounded, mean-zero random variables, and set
\[
   \epsilon_{k} = \frac{\pi}{N}+\frac{\eta_{k}}{N^{1+\delta}}\quad (\delta>0), 
   \qquad
   f_k(z) = \frac{z}{1-z} + \epsilon_k^2 .
\]
Then
\[
   f_N\circ f_{N-1}\circ \cdots\circ f_1 \longrightarrow \mathrm{Id}
\]
as \(N\to\infty\), uniformly on compact subsets of $\mathbb{C}$, with probability one.
\end{mainthm}

\begin{proof}
For $\epsilon_k$ as above, we have:
\[
2-\epsilon_k^2=2-\frac{\pi^2}{N^2}-\frac{2\pi\eta_k}{N^{2+\delta}}-\frac{\eta_k^2}{N^{2+2\delta}}.
\]
For $x=2\cos(\pi/N)$, we have
$x+d_k=2-\epsilon_k^2$ when
\[
d_k=-\frac{2\pi\eta_k}{N^{2+\delta}}-\frac{\eta_k^2}{N^{2+2\delta}}+O(1/N^4).
\]
We see that this satisfies conditions of equation \ref{randomerror}.
For the composition of the linear transformations, we use the 
Lemma~\ref{lem:recursivepoly} 
\[
F_N(z)=\frac{A_N z+B_N}{C_N z+D_N},
\]
with
\[
A_N=q_{N+1}-q_N,\quad
B_N=r_N-r_{N+1},\quad
C_N=-q_N,\quad
D_N=r_N.
\]
Using Lemma~\ref{lemma:skEpsilon} and Proposition~\ref{prop:randombigO}, we obtain
\[
A_N\to -1,\quad B_N\to 0,\quad C_N\to 0,\quad D_N\to -1
\]
almost surely.  
Thus $F_N\to\mathrm{Id}$ uniformly on compact sets.

Since the failure probabilities are summable, Borel--Cantelli implies that these bounds hold for all sufficiently large $N$ almost surely.
\end{proof}
 
In the random setting, independence and zero mean suffice to ensure almost sure convergence. This suggests that non-autonomous parabolic implosion may be robust under stochastic perturbations, including in higher-dimensional settings.

\section*{Acknowledgements}
The authors thank Hoi Nguyen for suggesting Azuma's lemma as a tool to prove convergence in the random setting. SS was partially supported by NSF DMS-2231565, and LV was partially supported by NSF 2453797.

\bibliographystyle{plain}
\bibliography{main}
\nocite{*}

\end{document}